\documentclass{lmsgai}
\usepackage{amssymb,amsmath,amscd,txfonts,mathrsfs}
\usepackage[perpage,symbol]{footmisc}

\newtheorem{theorem}{Theorem}[section] 
\newtheorem{lemma}[theorem]{Lemma}     
\newtheorem{corollary}[theorem]{Corollary}
\newtheorem{proposition}[theorem]{Proposition}

\newnumbered{assertion}{Assertion}    
\newnumbered{conjecture}{Conjecture}  
\newnumbered{definition}{Definition}
\newnumbered{hypothesis}{Hypothesis}
\newnumbered{remark}{Remark}
\newnumbered{note}{Note}
\newnumbered{observation}{Observation}
\newnumbered{problem}{Problem}
\newnumbered{question}{Question}
\newnumbered{algorithm}{Algorithm}
\newnumbered{example}{Example}
\newunnumbered{notation}{Notation} 

\title[Cubic hypersurfaces that split off two forms]
 {Rational points on cubic hypersurfaces that split off two forms} 

\author{Boqing Xue and Haobo Dai}

\classno{11D72 (primary), 11E76, 11P55, 14G25 (secondary)}

\extraline{The first author is supported by the National Natural Science Foundation of China (Grant No. 11271249).}

\begin{document}
\maketitle

\begin{abstract}
We show that if $X\subseteq \mathbb{P}^{n-1}$, defined over $\mathbb{Q}$ by a cubic form that splits off two forms, with $n\geq 11$, then $X(\mathbb{Q})$ is non-empty. The same holds for an $(m_1,m_2)$-form with $m_1\geq 4$ and $m_2\geq 5$.
\end{abstract}


\section{Introduction} 
\label{intro}

Let $X\subseteq \mathbb{P}^{n-1}$ be a cubic hypersurface, defined by $C=0$ with $C\in \mathbb{Z}[x_1,\ldots,x_n]$ a cubic form. It is conjectured that $X(\mathbb{Q})\neq \emptyset$ whenever $n\geq 10$. Local obstacles may exist for cubic forms in less variables. Mordell\cite{Mor} gave a counterexample for $n=9$. For more general results, we refer the readers to Birch\cite{Birch61}.

It was shown by Davenport \cite{Dav63} that an arbitrarily cubic surface has a $\mathbb{Q}$-point when $n\geq 16$. Heath-Brown \cite{HB07} improved this unconditional result to $n\geq 14$. Lewis and Birch also gave early results. Recently, Browning \cite{Browning1} showed that $X(\mathbb{Q})\neq \emptyset$ if $n\geq 13$ and $C$ splits off a form. Here $C$ splits off an $m_1$-form, or $C$ is an $(m_1,m_2)$-form, means that
\begin{equation}
C(x_1,\ldots,x_n)=C_1(y_1,\ldots,y_{m_1})+C_2(z_1,\ldots,z_{m_2}), \label{split1}
\end{equation}
with $m_1+m_2=n$, $m_1,m_2\geq 1$ and $C_1,C_2$ non-zero cubic forms with integer coefficients. And we say $C$ splits off a form if $C$ splits off an $m_1$-form for some $0<m_1<n$. He also suggested that cubic hypersurfaces that split off two forms be investigated, where we say $C$ splits off two forms if $C_1,C_2,C_3$ are non-zero cubic forms with integer coefficients and
\begin{equation}
C(x_1,\ldots,x_n)=C_1(w_1,\ldots,w_{n_1})+C_2(y_{1},\ldots,y_{n_2})+C_3(z_{1},\ldots,z_{n_3}), \label{split2}
\end{equation}
for appropriate $n_1, n_2, n_3\geq 1$ with $n_1+n_2+n_3=n$. We also call it an $(n_1,n_2,n_3)$-form. In this paper, we first establish the following theorem.
\begin{theorem} \label{TH1}
Let $X\subseteq \mathbb{P}^{n-1}$ be a hypersurface defined by a cubic form that splits off two forms, with $n\geq 11$. Then $X(\mathbb{Q})\neq \emptyset$.
\end{theorem}
Requiring more on the hypersurface, the folklore conjecture has been verified true.
 We say a cubic hypersurface $X$ non-singular if over $\overline{\mathbb{Q}}^n$ the only solution to the system of equations $\nabla C(\mathbb{\mathbf{x}})=0$ is $\mathbf{x}=0$. Heath-Brown \cite{HB83} showed that for non-singular cubic forms $n\geq 10$ variables are enough to guarantee a $\mathbb{Q}$-point. An extended version by Browning and Heath-Brown\cite{Br-HB09} shows that the condition $n\geq 10$ can be replaced by $n\geq 11+\sigma_X$, where $\sigma_X$ denotes the dimension of the singular locus of $X$.

For non-singular cubic forms in no more than 9 variables, it is expected that the Hasse principle still holds as soon as $n\geq 5$, which means that $X(\mathbb{Q})\neq \emptyset$ provided that $X(\mathbb{R})\neq \emptyset$ and $X(\mathbb{Q}_p)\neq \emptyset$ for every prime $p$. Hooley studied nonary cubic forms in a series of papers \cite{Hoo88}-\cite{Hoo12}. He first proved that Hasse principle holds for non-singular $X$ whenever $n\geq 9$. And most recently, he showed the following theorem.
\begin{theorem*}[H ({\cite{Hoo12}})] \label{THH}
Let $X\subseteq \mathbb{P}^{8}$ be a cubic hypersurface defined over $\mathbb{Q}$. Suppose that $X$ possesses at most isolated ordinary (double) points as singularities and $X(\mathbb{Q}_p)\neq \emptyset$ for every prime $p$. Then $X(\mathbb{Q})\neq \emptyset$.
\end{theorem*}
For singular cubic hypersurfaces $X$, Colliot-Th\'{e}l\`{e}ne and Salberger \cite{Co-Sal89} proved that the Hasse principle holds if $X$ contains a set of three conjugate singular points.
\begin{theorem*}[CS (\cite{Co-Sal89})] \label{THCS}
Let $X\subseteq \mathbb{P}^{n-1}$ be a cubic hypersurface defined over $\mathbb{Q}$, with $n\geq 4$. Suppose that $X$ contains a set of three conjugate singular points and $X(\mathbb{Q}_p)\neq \emptyset$ for every prime $p$. Then $X(\mathbb{Q})\neq \emptyset$.
\end{theorem*}
The structure of hypersurfaces defined by cubic forms with few variables are not hard to determine.
 With some geometric lemmas, it is given in \cite[Theorem 2]{Browning1} that $X(\mathbb{Q})\neq \emptyset$ if $C$ splits off an $m_1$-form with $m_1\geq 8$ and $n\geq 10$.
  And Browning\cite{Bro12} has shown us that the condition $m_1\geq 8$ can be replaced by $m_1\geq 5$.
   Based on his arguments, the following conclusions can be established.
\begin{theorem} \label{TH2}
Let $X\subseteq \mathbb{P}^{n-1}$ be a cubic hypersurface defined by an $(m_1,m_2)$-form $C$, with $m_1+m_2=n$. Suppose that $C$ has shape (\ref{split1}).

(\romannumeral1) If $C_1$ is non-singular, $m_1\geq 4$, $n\geq 9$ and $(m_1,m_2)\neq (6,3)$, then $X(\mathbb{Q})\neq \emptyset$.

(\romannumeral2) If $m_1\geq 4$ and $m_2\geq 5$, then $X(\mathbb{Q})\neq \emptyset$.
\end{theorem}
\begin{corollary}
Let $X\subseteq \mathbb{P}^{n}$ be a cubic hypersurface defined by an $(n_1,n_2,n_3)$-form, with $n_1+n_2+n_3\in \{9,10\}, ~ n_1\leq n_2\leq n_3$, $(n_1,n_2)\notin \{(1,1),(1,2)\}$ and $(n_1,n_2,n_3)\neq (3,3,3)$. Then $X(\mathbb{Q})\neq \emptyset$.
\end{corollary}
The local conditions may fail for a $(3,3,3)$-form. Consider
\begin{equation}
(x_1^3+2x_2^3+4x_3^3+x_1x_2x_3)+7(x_4^3+2x_5^3+4x_6^3+x_4x_5x_6)+49(x_7^3+2x_8^3+4x_9^3+x_7x_8x_9). \label{counterexample}
\end{equation}
The only solution to $x_1^3+2x_2^3+4x_3^3+x_1x_2x_3\equiv 0 \, (\text{mod }7)$ is $x_1,x_2,x_3\equiv 0 \, (\text{mod }7)$. So (\ref{counterexample}) does not represent zero non-trivially in $\mathbb{Q}_7$, and in $\mathbb{Q}$. See \cite{Mor} for more general counterexamples.

Moreover, we say $C$ captures $\mathbb{Q}^\ast$ if $C$ represents all the non-zero $r=a/q \in \mathbb{Q}$, using rational values for the variables. Fowler \cite{Fow} showed that any non-degenerate cubic form, by which we mean that it is not equivalent over $\mathbb{Z}$ to a cubic form in fewer variables, in no less than 3 variables that represents zero automatically captures $\mathbb{Q}^\ast$. For $C$ that can't represents zero non-trivially, we have $C$ captures $\mathbb{Q}^\ast$ if
\begin{displaymath}
q C(x_1,\ldots,x_n)-a x_{n+1}^3=0
\end{displaymath}
always has non-zero solutions for any integers $q$ and $a\neq 0$. Noting that the above equation involves a cubic form that splits off a $1$-form. Then Theorem \ref{TH1} directly implies the following.
\begin{corollary}
Let $C\in\mathbb{Z}[x_1,\ldots,x_n]$ be a non-degenerate cubic form that splits off a form, with $n\geq 10$. Then $C$ captures $\mathbb{Q}^\ast$.
\end{corollary}
We mainly follow the argument of \cite{Browning1}. Circle method is used. Note that the target forms can be reduced to forms in less variables if some of the variables take value 0, and forms of shape (\ref{split2}) can also be regarded as forms of shape (\ref{split1}). To prove Theorem \ref{TH1} and \ref{TH2}, it is sufficient to handle forms with type
\begin{equation}
(n_1,n_2,n_3)=(1,1,9),(1,2,8), \quad (m_1,m_2)=(4,5). \label{cases}
\end{equation}
After studiously calculating, one can get a weaker version of Theorem \ref{TH1}, with $n\geq 11$ replaced by $n\geq 12$. To save another variable needs two additional ingredients. Since exponential sums in many variables are harder to understand than that in a single variable, minor arc estimates fail for cubic forms that split off an $m_1$-form with $m_1\geq 3$. Geometric points of view (especially Theorem \ref{TH2}) do help in these cases. Another difficulty comes from some $(n_1,n_2,n_3)$-forms with $n_1,n_2\leq 2$. Although there exists many estimates that may be useful, the saving derived from $1$ or $2$-forms can't be as much as we want due to the small number of variables. To make enough saving in the case (1,1,9), Br\"{u}dern's result on a certain fourth moment of a cubic exponential sum is needed.

The remainder of this paper is laid out as follows. In \S2 geometry of singular cubic hypersurfaces is quoted and the proof of Theorem \ref{TH2} is given. In \S3 the circle method is introduced. In \S4 analytic results on exponential sums are listed. And two technical lemmas (Lemma \ref{LE7} and Lemma \ref{LE8}) are put into a type that can be computed and verified by computer easily. In \S5, we bound the minor arc estimates and the rest of the proof is given. At last in \S6, we make some further remarks that show the difficulty of improving $n\geq 11$ to $n\geq 10$ in Theorem \ref{TH1}.

Throughout this paper, parameters $\varepsilon,\delta,\Delta$ are carefully chosen small positive numbers satisfying $0<\varepsilon<\delta<\Delta$. For a point $\mathbf{x}=(x_1,\ldots,x_n)\in \mathbb{R}^n$, the norm $|\mathbf{x}|=\max_{1\leq i\leq n}|x_j|$. Symbols $\ll$, $\gg$. $\asymp$, $\textit{O}(\cdot)$, $\textit{o}(\cdot)$ are Vinogradov notations.

\section{Geometric results on singular cubic hypersufaces}

We use $C\in \mathbb{Z}[x_1,\ldots,x_{n}]$ to denote an arbitrary cubic form that defines $X\subseteq \mathbb{P}^{n-1}$. If $C$ has the shape (\ref{split1}). We denote $X_i\subseteq \mathbb{P}^{n_i-1}$ the variety defined by $C_i$ $(i=1,2)$, respectively. If $C$ splits off a form $C_1$ and $C_1=0$ has a non-trivial rational solution, then obviously $X(\mathbb{Q})$ is non-empty. So in the rest of this section we always assume that any form split off by $C$ does not have non-trivial rational solutions.

Theorem H and Theorem CS have given conditions under which Hasse principle holds. Now we investigate when local conditions hold.
 It is shown by Heath-Brown\cite[Proposition 2]{HB83} that any nonary cubic form defined over $\mathbb{Q}_p$ that splits off a $1$-form represents zero non-trivially in $\mathbb{Q}_p^9$.
  Some more effort leads to the following.
\begin{proposition} \label{PR1}
Let $C(x_1,\ldots,x_9)$ be a nonary cubic form over $\mathbb{Q}_p$. If $C$ splits off an $m$-form, with $m\in [1,8]\setminus \{3,6\}$, then it represents zero non-trivially in $\mathbb{Q}_p^9$.
\end{proposition}
\begin{proof}
Suppose $C$ has shape (\ref{split1}). By \cite[Proposition 1]{HB83}, either $C_1$ represents zero or there is a non-singular linear transformation sending $C_1$ to a form
\begin{displaymath}
D_1(u_1,\ldots,u_{r_1},v_1,\ldots,v_{r_2},w_1,\ldots,w_{r_3})=F_1(\mathbf{u})+pF_2(\mathbf{v})+p^2 F_3(\mathbf{w})+p G(\mathbf{u},\mathbf{v},\mathbf{w}),
\end{displaymath}
where $r_1,r_2,r_3\geq 0$ and $r_1+r_2+r_3=m$. And $F_1,F_2,F_3,G$ are forms with coefficients in $\mathbb{Z}$ with the following properties.

(\romannumeral1) $F_1$ involves the variables $u_i$ only, and similar for $F_2,F_3$. Terms involving the variables $u_i$ only (or the $v_i$ only, or the $w_i$ only) are absent from $G$.

(\romannumeral2) The congruence $F_1(\mathbf{u})\equiv 0 (\text{mod }p)$ has only the solution $\mathbf{u}\equiv \mathbf{0} (\text{mod }p)$, and similarly for $F_2,F_3$.

(\romannumeral3) In $G$, the coefficients of the monomials $u_iw_jw_k, v_iw_jw_k,w_iv_jv_k$ (where $i,j,k$ need not be distinct) are multiples of $p$.

Similarly, if $C_2$ does not represent zero, we write $C_2$ as
\begin{displaymath}
D_2(u_1^\prime,\ldots,u_{s_1}^\prime,v_1^\prime,\ldots,v_{s_2}^\prime,w_1^\prime,\ldots,w_{s_3}^\prime)=F_1^\prime(\mathbf{u}^\prime)+pF_2^\prime(\mathbf{v}^\prime)+p^2 F_3^\prime(\mathbf{w}^\prime)+p G^\prime(\mathbf{u}^\prime,\mathbf{v}^\prime,\mathbf{w}^\prime),
\end{displaymath}
where $s_1,s_2,s_3\geq 0$ and $s_1+s_2+s_3=9-m$. And $F_1^\prime,F_2^\prime,F_3^\prime,G$ are forms with coefficients in $\mathbb{Z}$ with similar properties. Hence $C$ is equivalent to
\begin{displaymath}
\left(F_1(\mathbf{u})+F_1^\prime(\mathbf{u}^\prime)\right)+p\left(F_2(\mathbf{v})+F_2^\prime(\mathbf{v}^\prime)\right)+p^2 \left( F_3(\mathbf{w})+ F_3^\prime(\mathbf{w}^\prime)\right)+p \left( G(\mathbf{u},\mathbf{v},\mathbf{w})+ G^\prime(\mathbf{u}^\prime,\mathbf{v}^\prime,\mathbf{w}^\prime)\right).
\end{displaymath}

By Chevalley's Theorem (see \cite[p.5]{Ser}) and the property (\romannumeral2), $D_1$ represents zero non-trivially unless $r_1,r_2,r_3\leq 3$. Similarly we can assume that $s_1,s_2,s_3\leq 3$. Since $m$ and $9-m$ do not take value $3$ or $6$, it can be deduced that there exists some $i_0\in\{1,2,3\}$ so that the value of the couple $(r_{i_0}, s_{i_0})$ belongs to $\{(2,1),(1,2),(3,1),(1,3),(2,3),(3,2),(2,2),(3,3)\}$. Without loss of generality, we suppose $(r_{i_0}, s_{i_0})=(2,1)$. (In other cases we can set the redundant variables 0 or change the order. And property (\romannumeral2) ensures that if appropriate variables are chosen, the induced form will not be identically zero.) Now $F_{i_0}+F^\prime_{i_0}$ is the sum of a $2$-form and a $1$-form, and the arguments of \cite[Proposition 2]{HB83} leads to the conclusion that $F_{i_0}+F^\prime_{i_0}$ represents zero non-trivially in $\mathbb{Q}_p^9$. Let the variables that do not appear in $F_{i_0}$ and $F^\prime_{i_0}$ be zero. From the second statement of the property (\romannumeral1), it can be asserted that the terms $G$ and $G^\prime$ vanish. Hence $C$ also has non-trivial $\mathbb{Q}_p$- solutions.
\end{proof}
%
%
Combing Theorem H, Theorem CS and Proposition \ref{PR1}, we obtain the following corollaries.

%
%
\begin{corollary} \label{CO1}
Let $X\subseteq \mathbb{P}^{8}$ be a cubic hypersurface defined by a nonary cubic form that splits off an $m$-form, with $m\in [1,8]\setminus \{3,6\}$. Suppose that $X$ possesses at most isolated ordinary (double) points as singularities. Then $X(\mathbb{Q})\neq \emptyset$.
\end{corollary}
\begin{corollary} \label{CO2}
Let $X\subseteq \mathbb{P}^{8}$ be a cubic hypersurface defined by a nonary cubic form that splits off an $m$-form, with $m\in [1,8]\setminus \{3,6\}$. Suppose that $X$ contains a set of three conjugate singular points. Then $X(\mathbb{Q})\neq \emptyset$.
\end{corollary}

Given a cubic extension $K$ of $\mathbb{Q}$, define the corresponding norm form
\begin{equation}
N(x_1,x_2,x_3):=N_{K/\mathbb{Q}}(\omega_1x_1+\omega_2x_2+\omega_3x_3),
\end{equation}
where $\{\omega_1,\omega_2,\omega_3\}$ is a basis of $K$ as a vector space over $\mathbb{Q}$.

For small $n$, the geometry of $X\subseteq \mathbb{P}^{n-1}$ defined by a cubic form is not hard to determine. The following lemma collects \cite[Lemma 2-4]{Browning1}.
\begin{lemma} \label{LE1}
Suppose $X\subseteq \mathbb{P}^{n-1}$ is defined by $C=0$ with $C\in \mathbb{Z}[x_1,\ldots,x_n]$ a cubic form and $X(\mathbb{Q})=\emptyset$.

(\romannumeral1) For $n=3$, either the curve $X$ is non-singular or $X$ contains precisely three conjugate singular points. In particular in the latter case, $C$ can be written as a norm form, i.e.,
\begin{displaymath}
C(\mathbf{x})=\text{N}_{K/\mathbb{Q}}(x_1\omega_1+x_2\omega_2+x_3\omega_3),
\end{displaymath}
for some appropriate coefficients $\omega_1,\omega_2,\omega_3\in K$, where $K$ is the cubic number field obtained by adjoining one of the singularities.

(\romannumeral2) For $n=4$, either the surface $X$ is non-singular or $X$ contains precisely three conjugate double points. In particular in the latter case we have the representation
\begin{displaymath}
C(\mathbf{x})=\text{N}_{K/\mathbb{Q}}(x_1\omega_1+x_2\omega_2+x_3\omega_3)+ax_4^2\text{Tr}_{K/\mathbb{Q}}(x_1\omega_1+x_2\omega_2+x_3\omega_3)+bx_4^3,
\end{displaymath}
for some appropriate coefficients $\omega_1,\omega_2,\omega_3\in K$ and $a,b\in \mathbb{Z}$.

(\romannumeral3) For $n=5$, either the threefold $X$ is non-singular or $X$ is a geometrically integral cubic hypersurface whose singular locus contains precisely $\delta$ double points, with $\delta\in \{3,6,9\}$.
\end{lemma}
%

\begin{proof}[of Theorem \ref{TH2}]
We first prove Theorem \ref{TH2}(\romannumeral2). Suppose that $C$ has the shape (\ref{split1}). It is sufficient to handle the case $(m_1,m_2)=(4,5)$. By Lemma \ref{LE1}(\romannumeral2), either $X_1$ is non-singular or $X_1$ contains precisely three conjugate singular points. In the former case, $X_2$ has at most isolated double points (by Lemma \ref{LE1}(\romannumeral3)) and so does $X$. Since $C$ splits off a $4$-form. Corollary \ref{CO1} ensures that $X(\mathbb{Q})\neq \emptyset$. In the latter case, $X$ contains a set of three conjugate singular points and $C$ splits off a $4$-form. Corollary \ref{CO2} shows that $X(\mathbb{Q})\neq \emptyset$ also holds.

As for Theorem \ref{TH2}(\romannumeral1), the cases $m_1=4$ or $5$ can be implied from Theorem \ref{TH2}(\romannumeral2). The cases $m_1\geq 9$ can be deduced from \cite[Theorem 2]{Browning1}. For $m_1=7$ or $8$. It is sufficient to handle $(m_1,9-m_1)$, i.e., we put $m_2=9-m_1$. Then $m_2=1$ or $2$. Since $C_1$ is non-singular, it follows that $C$ is a non-singular cubic form in 9 variables that splits off an $m_2$-form. By Corollary \ref{CO1}, it can be conclude that $X(\mathbb{Q})\neq \emptyset$. At last we suppose $m_1=6$ and $m_2\geq 4$. This case is dealt with in \cite{Bro12}. It can also be reduced to the case $(5,4)$, which is implied by Theorem \ref{TH2}(\romannumeral2).
\end{proof}

%
\section{The circle method}

For Theorem \ref{TH1} it is suffice to handle the case $n_0=11$, since when $n_0>11$ we can simply force the redundant variables to be $0$. Most of the results outlined in \S1 involve the circle method in the proof. We apply it to deal with $(1,1,9)$ and $(1,2,8)$-forms.

We use $C_0\in \mathbb{Z}[x_1,\ldots,x_{n_0}]$ to denote the cubic form stated in Theorem \ref{TH1}, which defines $X_0\subseteq \mathbb{P}^{n_0-1}$. Let $n_1,n_2,n_3\geq 1$ be integers such that $n_1+n_2+n_3=n_0$. Since $C_0$ splits off two forms, we write
\begin{equation}
C_0(\mathbf{x})=C_1(\mathbf{x}_1)+C_2(\mathbf{x}_2)+C_3(\mathbf{x}_3), \label{split}
\end{equation}
where $C_i\in \mathbb{Z}[\mathbf{x}_i]\,(1\leq i\leq 3)$ are cubic forms in $n_i$ variables and define $X_i\subseteq \mathbb{P}^{n_i-1}$, respectively. Without loss of generality, we assume that $n_1\leq n_2\leq n_3$.

If $C$ is degenerate, then $C=0$ has obvious non-zero integer solutions. If $C$ has no less than 3 variables and is not `good', then $C=0$ also has non-zero integer solutions for `geometric reasons' (see \cite{Dav}). Here a general cubic form $C$ is `good' means that for any $H\geq 1$ and any $\varepsilon>0$, the upper bound
\begin{displaymath}
\#\{\mathbf{x}\in \mathbb{Z}^n:|\mathbf{x}|\leq H, \text{rank}H(\mathbf{x})=r\}\ll H^{r+\varepsilon}
\end{displaymath}
holds for each integer $0\leq r\leq n$, where $H(\mathbf{x})$ is the Hessian matrix of $C$. Any cubic form defining a hypersurface
with at most isolated ordinary singularities is good, which is due to Hooley\cite{Hoo88}. Moreover, if $C_i=0$ for some $1\leq i\leq 3$ or $C_i+C_j=0$ for some $1\leq i<j\leq 3$ has non-trivial integer solutions, then we easily see that $C_0=0$ has non-trivial integer solutions. Now we can suppose that none of $C_j$ $(0\leq j\leq 3)$ and $C_i+C_j$ $(1\leq i<j\leq 3)$ has integer solutions, is degenerate, or is not `good' whenever no less than 3 variables are possessed.

Write $e(x):=e^{2\pi i x}$. Define the cubic exponential sum
\begin{displaymath}
S(\alpha)=S(\alpha;C,n,\rho,P):=\sum\limits_{\mathbf{x}\in\mathbb{Z}^n\atop |P^{-1}\mathbf{x}-\mathbf{z}|< \rho} e(\alpha C(x)),
\end{displaymath}
where $\mathbf{z}$ is a fixed vector and $\rho>0$ is a fixed real number, both to be determined later. The precise value of $\mathbf{z}$ and $\rho$ are actually immaterial and the corresponding implied constants are allowed to depend on these quantities. Let $S_i(\alpha):=S(\alpha;C_i,n_i,\rho,P)$ for $0\leq i\leq 3$. From (\ref{split}), one has
\begin{displaymath}
S_0(\alpha)=S_1(\alpha)S_2(\alpha)S_3(\alpha).
\end{displaymath}
Write
\begin{displaymath}
\mathcal{N}(P):=\#\{\mathbf{x}\in \mathbb{Z}^{n_0}: |P^{-1}\mathbf{x}-\mathbf{z}|< \rho, \; C_0(\mathbf{x})=0\}.
\end{displaymath}
On observing the simple equality
\begin{displaymath}
\int_0^1 e(\alpha x)d\alpha =
\begin{cases}
1,\quad \text{if }x=0,\\
0,\quad \text{if }x\neq 0,
\end{cases}
\end{displaymath}
the number of solutions of $C_0(\mathbf{x})=0$ counted by $\mathcal{N}(P)$ is exactly
\begin{displaymath}
\int_0^1{S_0(\alpha)}d\alpha.
\end{displaymath}
Next we divide the integral domain into two parts where different tools can be applied. Define the major arcs as
\begin{displaymath}
\mathfrak{M}:=\bigcup\limits_{q\leq P^\Delta}\bigcup\limits_{(a,q)=1}\left[\frac{a}{q}-P^{-3+\Delta},\frac{a}{q}+P^{-3+\Delta}\right]
\end{displaymath}
and the minor arcs as $\mathfrak{m}:=[0,1]\setminus \mathfrak{M}$, where $\Delta$ is a small positive integer to be specified later. The intervals in the major arcs are pairwise disjoint. The integral becomes
\begin{displaymath}
\int_0^1{S(\alpha)}d\alpha=\int_{\mathfrak{M}}{S(\alpha)}d\alpha+\int_{\mathfrak{m}}{S(\alpha)}d\alpha.
\end{displaymath}
If the first term on the right side (known as the main term) takes positive value and overwhelms the second term (known as the error term) for sufficiently large $P$, then we reach the declaration that $\mathcal{N}(P) \gg P^{\tau}$ ($\tau$ can be 8 according to Lemma \ref{LE2} below) and then $C_0=0$ has non-trivial integer solutions. As a result, we have $X_0(\mathbb{Q})\neq \emptyset$ and Theorem \ref{TH1} follows.

The following lemma ensures that the integral over major arcs are `large'.
\begin{lemma} \label{LE2}
Let $n_0=11$. We have
\begin{displaymath}
\int_\mathfrak{M} S_0(\alpha)d\alpha= \mathfrak{S}\mathfrak{J}P^{n_0-3}+\textit{o}\left(P^{n_0-3}\right),
\end{displaymath}
where
\begin{displaymath}
\mathfrak{S}:=\sum\limits_{q=1}^\infty \sum\limits_{(a,q)=1}q^{-n_0}S_{a,q}(C_0),\quad \mathfrak{J}:=\int_{-\infty}^\infty I(\beta;C_0) d\beta,
\end{displaymath}
with
\begin{displaymath}
S_{a,q}=S_{a,q}(C):=\sum\limits_{\mathbf{y} (\text{mod }q)} e_q(a C(\mathbf{y})),\quad I(\beta)=I(\beta;C):=\int_{|P^{-1}\mathbf{x}-\mathbf{z}|<\rho} e(\beta C(\mathbf{x}))d\mathbf{x}.
\end{displaymath}
\end{lemma}
The proof of Lemma \ref{LE2} uses standard arguments. One can see \cite[Lemma 15.4, \S16-18]{Dav} or \cite[Lemma 2.1]{HB07} for details. Since $C_0$ is good, Heath-Brown's bound (see \cite[(7.1)]{HB07}) $S_{a,q}(C_0)\ll q^{5n_0/6+\varepsilon}$ is effective and \cite[Theorem 4]{HB07} ensures that $\mathfrak{S}$ is absolutely convergent for $n_0=11$. Standard argument (see \cite[Lemma 7.3]{Dav59}) leads to $\mathfrak{S}>0$. Assuming that $C_3$ does not have a linear factor defined over $\mathbb{Q}$ (otherwise non-trivial integer solutions can be found easily), it is possible to choose appropriate $\mathbf{z}_3$ in the definition of $S_3$ so that $\mathbf{z}_3$ is a non-singular real solution to $C_3=0$. Now pick $\mathbf{z}_1=\mathbf{z}_2=\mathbf{0}$, then $\mathbf{z}=(\mathbf{z}_1,\mathbf{z}_2,\mathbf{z}_3)$ is a non-singular real solution to equation $C_0=0$. On selecting a sufficiently small value of $\rho>0$, we will have $\mathfrak{J}>0$. (See \cite[\S16]{Dav} and \cite[\S4]{HB83} for details.)

Now we only need to show that the integral over the minor arcs is `small' according to that over the major arcs.
\begin{proposition} \label{PR2}
Suppose that  $(n_1,n_2,n_3)\in \{(1,1,9),(1,2,8)\}$. Then either $X_i(\mathbb{Q})\neq \emptyset$ for some $1\leq i\leq 3$, or
\begin{displaymath}
\int_\mathfrak{m}S_{n_0}(\alpha)d\alpha=\textit{o}\left(P^{n_0-3}\right).
\end{displaymath}
\end{proposition}

In the latter case, we also have $X_0(\mathbb{Q})\neq \emptyset$ in view of Lemma \ref{LE2}. Hence Theorem \ref{TH1} follows from Proposition \ref{PR2} immediately.

Note that the integral over the minor arcs is a kind of $L^1$ norm. The $L^\infty$ bound on the exponential sums is the essential content of Davenport's result on cubic forms in 16 variables. Later Heath-Brown took use of Van der Corput's method, with two additional techniques, resulted in a powerful $L^2$ bound and he successfully worked out the solubility of cubic forms in 14 variables with no restrictions. Combining these two kind of bounds, Browning showed an $L^v$ bound with $v\leq 2$. For $n=1,2$, there are good $L^v$ bounds with $v\geq 2$ (Hua's inequality in dimension one and its analog in dimension two by Wooley). And we mention that for $n=3,4,5$, $L^2$ bounds resulted from the number of solutions of cubic forms are also available. We implant these analytic results in \S4.

To apply such tools, we need to select the appropriate combination of the powers in H\"{o}lder's inequality. For simplicity, we use the following notations:
\begin{displaymath}
I_u(S;t,\mathfrak{a}):=P^{t+\varepsilon}\left(\int_{\mathfrak{a}}|S(\alpha)|^u d\alpha\right)^{1/u},
\end{displaymath}
\begin{displaymath}
I_{u,v}(S_{n_1},S_{n_2};t,\mathfrak{a}):=P^{t+\varepsilon}\left(\int_{\mathfrak{a}}|S_{n_1}(\alpha)|^u d\alpha\right)^{1/u}\left(\int_{\mathfrak{a}}|S_{n_2}(\alpha)|^v d\alpha\right)^{1/v},
\end{displaymath}
\begin{displaymath}
I_{u,v,w}(S_{n_1},S_{n_2},S_{n_3};t,\mathfrak{a}):=P^{t+\varepsilon}\left(\int_{\mathfrak{a}}|S_{n_1}(\alpha)|^u d\alpha\right)^{1/u}\left(\int_{\mathfrak{a}}|S_{n_2}(\alpha)|^v d\alpha\right)^{1/v}\left(\int_{\mathfrak{a}}|S_{n_3}(\alpha)|^w d\alpha\right)^{1/w}.
\end{displaymath}

With these notations, Proposition \ref{PR2} can be implied from
\begin{displaymath}
I_{1}(S_1S_2S_3;0,\mathfrak{m})=\textit{o}\left(P^8\right).
\end{displaymath}

\section{Estimates on the cubic exponential sums}

\begin{lemma}[(\cite{Dav63})] \label{LE3}
Let $\varepsilon>0$. Assume that $C\in \mathbb{Z}[x_1,\ldots,x_n]$ is a good cubic form. Let $\alpha \in [0,1]$ have the representation
\begin{displaymath}
\alpha=a/q+\beta,\quad (a,q)=1,\quad 0\leq a<q\leq P^{3/2}
\end{displaymath}
with $a,q\in \mathbb{Z}$. Then
\begin{displaymath}
S(\alpha)\ll P^{n+\varepsilon}\left\{q|\beta|+\left(q|\beta|P^3\right)^{-1}\right\}^{n/8}.
\end{displaymath}
If furthermore $|\beta|\leq q^{-1}P^{-3/2}$, then
\begin{displaymath}
S(\alpha)\ll P^{n+\varepsilon}q^{-n/8}\min\left\{1,\left(|\beta|P^3\right)^{-n/8}\right\}.
\end{displaymath}
\end{lemma}
\begin{lemma}[({\cite[Lemma 7]{Browning1}})] \label{LE4}
Let $\varepsilon>0$. Assume that $C\in \mathbb{Z}[x_1,\ldots,x_n]$ is a good cubic form. Let $1\leq R\leq P^{3/2}$ and $0<\phi\leq R^{-2}$. Define
\begin{displaymath}
\mathcal{M}_v(R,\phi,\pm):=\sum\limits_{R\leq q< 2R} \sum\limits_{(a,q)=1} \int_\phi^{2\phi}\left|S\left(\frac{a}{q}\pm \beta\right)\right|^v d\beta.
\end{displaymath}
Then
\begin{displaymath}
\mathcal{M}_v(R\,\phi,\pm)\ll P^3+R^2\phi^{1-v/2}\left(\frac{\psi_H P^{2n-1+\varepsilon}}{H^{n-1}}F\right)^{v/2},
\end{displaymath}
with $H$ any integer in $[1,P]$ and
\begin{displaymath}
\psi_H:=\phi+\frac{1}{P^2 H},\quad F:=1+(RH^3\psi_H)^{n/2}+\frac{H^n}{R^{n/2}(P^2\psi_H)^{(n-2)/2}}.
\end{displaymath}
\end{lemma}
\begin{lemma} \label{LE5}
Let $\varepsilon>0$. The following bounds hold.

(\romannumeral1) (\cite[Lemma 3.2]{Dav}) For $n=1$, one has
\begin{displaymath}
\int_0^1|S(\alpha)|^{2^j}d\alpha \ll P^{2^j-j+\varepsilon},
\end{displaymath}
for any $j\leq 3$.

(\romannumeral2) (\cite[Theorem 2]{Woo99}) For $n=2$, suppose that $C\in \mathbb{Z}[x_1,x_2]$ is a non-degenerate binary cubic form, then
\begin{displaymath}
\int_0^1|S(\alpha)|^{2^{j-1}}d\alpha \ll P^{2^j-j+\varepsilon},
\end{displaymath}
for any $j\leq 3$.
\end{lemma}
For $A,B,C\geq 0$, define $\mathcal{A}=\mathcal{A}(A,B,C)$ to be the set of $\alpha\in [0,1]$ for which there exists $a,q\in \mathbb{Z}$ such that
\begin{displaymath}
\alpha=a/q+\beta,\quad (a,q)=1,\quad 1\leq a<q\leq P^A,\quad |\beta|\leq q^{-B}P^{-3+C}.
\end{displaymath}

For $n=1$, define
\begin{displaymath}
S^\ast (\alpha):= q^{-1} P\, S_{a,q} I(\beta P^3).
\end{displaymath}
We try to approximate $S(\alpha)$ by $S^\ast (\alpha)$. Since $S^\ast (\alpha)$ has better $L^v$ $(v\geq 2)$ bounds, we can gain extra saving if their difference is small in some particular intervals. The following lemma can be derived from the book of Vaughan \cite[\S4]{Vau97} (or see \cite[Lemma 10]{Browning1}).
\begin{lemma} \label{LE6}
Let $\varepsilon>0$ and $n=1$. Suppose $A,B,C\geq 0$ and $A,B\leq 1$. Then for any $\alpha\in \mathcal{A}$,
\begin{displaymath}
S(\alpha)=S^\ast(\alpha)+\textit{O}\left(P^{A/2+\varepsilon}+P^{(A+C-AB)/2+\varepsilon}\right).
\end{displaymath}
Furthermore, if $k\geq 4$, then
\begin{displaymath}
\int_\mathcal{A}|S^\ast(\alpha)|^k d\alpha \ll P^{k-3+\varepsilon}.
\end{displaymath}
\end{lemma}
\begin{lemma} \label{LE7}
Assume $\mathcal{A}$ is defined as above with $A,B,C\geq 0$. We have
\begin{displaymath}
I_v(S;t,\mathcal{A})\ll
\begin{cases}
P^{n+t-(3-C)\left(\frac{1}{v}+\frac{n}{8}\right)+A\left(\frac{2}{v}+\frac{n}{8}-B\left(\frac{1}{v}+\frac{n}{8}\right)\right)+\varepsilon}+
 P^{\frac{5}{8} n+t+\varepsilon+\min\left\{\frac{nA}{8} ,\;\; -(3-C)\left(\frac{1}{v}-\frac{n}{8}\right)+A\left(\frac{2}{v}-\frac{n}{8}-B\left(\frac{1}{v}-\frac{n}{8}\right)\right) \right\}},\quad \text{if }nv\leq 8.\\
P^{n+t-(3-C)\left(\frac{1}{v}+\frac{n}{8}\right)+A\left(\frac{2}{v}+\frac{n}{8}-B\left(\frac{1}{v}+\frac{n}{8}\right)\right)+\varepsilon}+
 P^{n+t-\frac{3}{v}+A\left(\frac{2}{v}-\frac{n}{8}\right)+\varepsilon},\quad \text{if } 8\leq nv\leq 16.\\
P^{n+t-(3-C)\left(\frac{1}{v}+\frac{n}{8}\right)+A\left(\frac{2}{v}+\frac{n}{8}-B\left(\frac{1}{v}+\frac{n}{8}\right)\right)+\varepsilon}+
 P^{n+t-\frac{3}{v}+\varepsilon},\quad \text{if }nv\geq 16.
\end{cases}
\end{displaymath}
\end{lemma}
\begin{proof}
By dyadic summation, we have
\begin{displaymath}
I_v(S;t,\mathcal{A})\ll P^{t+\varepsilon/2}(\log P)^2 \max\limits_{R,\phi,\pm} \mathcal{M}_v (R,\phi,\pm)^{1/v},
\end{displaymath}
where the maximum runs over the possible sign changes and $R,\phi$ satisfy $R\leq P^A$, $\phi\leq R^{-B}P^{-(3-C)}$.

For the case $R\phi\geq P^{-3/2}$, Lemma \ref{LE3} shows that $S(\alpha)\ll P^{n+\varepsilon}R^{n/8}\phi^{n/8}$. Then
\begin{displaymath}
I_v(S;t,\mathcal{A})\ll P^{t+\varepsilon}(R^2\phi)^{\frac{1}{v}}\cdot P^{n+\varepsilon}R^{\frac{n}{8}}\phi^{\frac{n}{8}}
  \ll P^{n+t+\varepsilon}R^{\frac{2}{v}+\frac{n}{8}}\phi^{\frac{1}{v}+\frac{n}{8}} \ll P^{n+t-(3-C)\left(\frac{1}{v}+\frac{n}{8}\right)+\varepsilon}R^{\frac{2}{v}+\frac{n}{8}-B\left(\frac{1}{v}+\frac{n}{8}\right)}.
\end{displaymath}
For $R\phi\leq P^{-3/2}$ and $nv\leq 8$, one has $S(\alpha)\ll P^{5n/8+\varepsilon}R^{-n/8}\phi^{-n/8}$. Then
\begin{displaymath}
I_v(S;t,\mathcal{A})\ll P^{t+\varepsilon}(R^2\phi)^{\frac{1}{v}}\cdot P^{\frac{5}{8} n+\varepsilon}R^{-\frac{n}{8}}\phi^{-\frac{n}{8}}
  \ll P^{\frac{5}{8} n+t+\varepsilon}R^{\frac{2}{v}-\frac{n}{8}}\phi^{\frac{1}{v}-\frac{n}{8}} \ll P^{\frac{5}{8} n+t-(3-C)\left(\frac{1}{v}-\frac{n}{8}\right)+\varepsilon}R^{\frac{2}{v}-\frac{n}{8}-B\left(\frac{1}{v}-\frac{n}{8}\right)}.
\end{displaymath}
Another estimate can be
\begin{displaymath}
I_v(S;t,\mathcal{A})\ll P^{\frac{5}{8} n+t+\varepsilon}R^{\frac{2}{v}-\frac{n}{8}}\phi^{\frac{1}{v}-\frac{n}{8}} =P^{\frac{5}{8} n+t+\varepsilon}R^{\frac{n}{8}}(R^2\phi)^{\frac{1}{v}-\frac{n}{8}}\ll P^{\frac{5}{8} n+t+\varepsilon}R^{\frac{n}{8}}\ll
P^{\frac{5}{8} n+t+\frac{nA}{8}+\varepsilon}.
\end{displaymath}
For $P^{-3}\leq \phi\leq R^{-1}P^{-3/2}$ and $nv\geq 8$, one has
\begin{equation}
I_v(S;t,\mathcal{A})\ll P^{\frac{5}{8} n+ t+\varepsilon}R^{\frac{2}{v}-\frac{n}{8}}\phi^{\frac{1}{v}-\frac{n}{8}} \ll P^{\frac{5}{8} n+t-3\left(\frac{1}{v}-\frac{n}{8}\right)+\varepsilon}R^{\frac{2}{v}-\frac{n}{8}}=P^{n+t-\frac{3}{v}+\varepsilon}R^{\frac{2}{v}-\frac{n}{8}}. \label{lem7}
\end{equation}
For $\phi\leq P^{-3}$ and $nv\geq 8$, one has $S(\alpha)\ll P^{n+\varepsilon}R^{-n/8}$. Then
\begin{equation}
I_v(S;t,\mathcal{A})\ll P^{t+\varepsilon}(R^2\phi)^{\frac{1}{v}}\cdot P^{n+\varepsilon}R^{-n/8}
  \ll P^{n+t+\varepsilon}R^{\frac{2}{v}-\frac{n}{8}}\phi^{\frac{1}{v}} \ll P^{n+t-\frac{3}{v}+\varepsilon}R^{\frac{2}{v}-\frac{n}{8}}. \label{lem7-2}
\end{equation}
Now Lemma \ref{LE7} follows.
\end{proof}
\begin{remark*}
When $nv> 16$, the exponent on $R$ is negative. And on $\mathfrak{m}$ we additionally have the fact that $R\leq P^{\Delta}$ and $\phi\leq P^{-3+\Delta}$ do not hold simultaneously. If $\phi\geq P^{-3+\Delta}$, the right side of (\ref{lem7}) can be replaced by $P^{\frac{5}{8} n+t+\frac{nA}{8}-\Delta(\frac{n}{8}-\frac{1}{v})+\varepsilon}R^{\frac{2}{v}-\frac{n}{8}}$. Otherwise we will have $R\geq P^{\Delta}$ and the right side of (\ref{lem7}), (\ref{lem7-2}) can be bounded by $P^{n+t-\frac{3}{v}-\Delta(\frac{n}{8}-\frac{2}{v})+\varepsilon}$. Hence actually we can achieve the bound
\begin{equation}
I_v(S;t,\mathcal{A}\cap \mathfrak{m})\ll P^{n+t-(3-C)\left(\frac{1}{v}+\frac{n}{8}\right)+A\left\{\frac{2}{v}+\frac{n}{8}-B\left(\frac{1}{v}+\frac{n}{8}\right)\right\}+\varepsilon}+
 P^{n+t-\frac{3}{v}-\Delta(\frac{n}{8}-\frac{2}{v})+\varepsilon}, \label{remarkoflemma7}
\end{equation}
provided that $nv>16$.
\end{remark*}
The next lemma is an extension of \cite[Lemma 14]{Browning1}. A series of parameters and conditions are listed first. We are sorry that readers may be confused by these parameters and conditions at first glance. They do shorten the proof and make the lemma convenient to use. Actually when we try to apply it in definite cases, they can be computed and verified easily by computers.

\medskip

Parameters:
\begin{gather*}
\rho_0:=\frac{2}{n},\quad \pi_0:=\frac{-2\Lambda+2t+4n-3}{n},\\
\rho_1:=\frac{n(n-5)}{n^2 -5n+2},\quad \pi_1:=\frac{-2(n^2-2n(\Lambda-t-1)-2)}{n^2 -5n+2},\\
\rho_2:=\frac{n-8}{n-4},\quad \pi_2:=\frac{8\Lambda-5n-8t}{n-4},\\
\Upsilon:=\frac{-6\Lambda +6t+6n-3}{n-1},\quad Q:=P^{3\varepsilon}\left(1+P^{\Upsilon}\right),\\
\Xi:=\frac{-\pi_1+\pi_0}{\rho_1-\rho_0}=\frac{(3n-2)(-2\Lambda+2t+2n-3)}{n^2-6n+4},\\
\phi_0=R^{-\rho_0}P^{-\pi_0},\quad \phi_1=R^{-\rho_1}P^{-\pi_1},\quad \phi_2=R^{-\rho_2}P^{-\pi_2}\\
\end{gather*}
Conditions:
\begin{eqnarray}
&3/2-\Upsilon>0,  \label{req0}\\
&\Lambda-t-3/2> 0, \label{req1}\\
&\Lambda-t-n/2> 0,  \label{req2}\\
&2\Lambda-2t-n-2\Upsilon+3>0, \label{req3}\\
&10 \Lambda -10 t- 8 n+3 \geq 0, \label{req7}\\
&\Lambda-\left(\frac{2}{v}+\frac{n}{8}-\rho_1\left(\frac{1}{v}+\frac{n}{8}\right)\right)\cdot \Xi -n-t+\pi_1\left(\frac{1}{v}+\frac{n}{8}\right)>0, \label{req11}\\
&\pi_2-3\geq 0. \label{req12}
\end{eqnarray}



The condition (\ref{req0}) ensures that the $R$ we take into consideration always satisfies $R\leq P^{3/2}$ and then Lemma \ref{LE4} can be applied. By Dirichlet's approximation theorem, for any $\alpha\in [0,1]$ there exists integers $a$ and $q$ such that
\begin{equation}
\alpha=a/q+\beta,\quad 1\leq a\leq q\leq Q,\quad (a,q)=1,\quad |\beta|\leq 1/(qQ). \label{apr}
\end{equation}
By dyadic summation, we have
\begin{displaymath}
I_v(S;t,\mathfrak{a})\ll P^{t+\varepsilon/2}(\log P)^2 \max\limits_{R,\phi,\pm} \mathcal{M}_v (R,\phi,\pm)^{1/v},
\end{displaymath}
where the maximum runs over the possible sign changes and $R,\phi$ are in the range described by $\mathfrak{a}$ and satisfy
\begin{displaymath}
1\leq R\leq Q,\quad 0< \phi\leq (RQ)^{-1}.
\end{displaymath}
A direct deduction shows that
\begin{equation}
1\leq R\leq Q,\quad R\phi \leq Q^{-1},\quad R^2\phi\leq 1. \label{apr_use}
\end{equation}
\begin{lemma} \label{LE8}
Let $n\geq 6$. Denote $\mathfrak{m}_0$ the set of $\alpha \in \mathfrak{m}$ with the representation (\ref{apr}) with
\begin{displaymath}
q\leq P^{\Xi+c\delta},\quad |\alpha-a/q|\leq \phi_2P^{\delta},
\end{displaymath}
where $c$ is a positive constant depending only on $n$. Then

(\romannumeral1) we have $I_2(S;t,\mathfrak{m}\setminus \mathfrak{m}_0)=\textit{o}\left(P^\Lambda\right)$, provided that (\ref{req0})-(\ref{req11}) holds.

(\romannumeral2) we have $I_2(S;t,\mathfrak{m})=\textit{o}\left(P^\Lambda\right)$, provided that $\Xi \leq 0$ and (\ref{req0})-(\ref{req12}) holds.
\end{lemma}
We prove it through the following two lemmas. The constant in the expression $\textit{O}(\delta)$ occurring in the proof of this lemma only depends on $n$.
\begin{lemma} \label{CL1}
Suppose that (\ref{req0}), (\ref{req1}), (\ref{req2}) holds. We have
\begin{displaymath}
I_2(S;t,\mathfrak{n}_1)=\textit{o}\left(P^\Lambda\right),
\end{displaymath}
where $\mathfrak{n}_1$ is the set of $\alpha\in \mathfrak{m}$ with the representation (\ref{apr}) with
\begin{displaymath}
q\leq Q,\quad \max\{\phi_0,\phi_1 P^\delta\}\leq |\alpha-a/q|\leq (qQ)^{-1}.
\end{displaymath}
\end{lemma}
\begin{proof}
It follows from Lemma \ref{LE4} that
\begin{displaymath}
I_2(S;t,\mathfrak{n}_1)\ll P^{t+\varepsilon}\left\{P^{3/2}+\max\limits_{1\leq R\leq Q \atop \max\{\phi_0,\phi_1P^\delta\}\leq \phi\leq (RQ)^{-1}}R\left(\frac{\psi_H P^{2n-1}}{H^{n-1}}F\right)^{1/2}\right\}.
\end{displaymath}
The first term on the right side is $\textit{O}(P^{\Lambda})$ whenever (\ref{req1}) holds. Now we estimate the second term
\begin{displaymath}
E_1:=P^{t+\varepsilon}\max\limits_{R,\phi}R\left(\frac{\psi_H P^{2n-1}}{H^{n-1}}F\right)^{1/2}.
\end{displaymath}

Take $H=\displaystyle \left\lfloor P^\varepsilon \left(1+R\phi^{1/2}P^{-\Lambda+t+n-1/2}\right)^{2/(n-1)}\right\rfloor$, then $\psi_H\asymp \phi$ when $\phi\geq \phi_0$. Combining (\ref{apr}) and (\ref{req2}), one has
\begin{displaymath}
H\leq P^\varepsilon\left(1+ (R^2\phi)^{1/(n-1)}P^{(-2\Lambda+2t+2n-1)/(n-1)}\right)\leq P.
\end{displaymath}
Hence the choice of $H$ is appropriate. Recall that $F=1+(RH^3\psi_H)^{n/2}+H^n R^{-n/2}P^{-(n-2)}\psi_H^{-(n-2)/2}$. Calculation reveals that
\begin{displaymath}
RH^3\psi_H\ll R\phi P^{3\varepsilon}\left\{1+\left((R^2\phi)^{1/2}P^{-\Lambda+t+n-1/2}\right)^{6/(n-1)}\right\}\ll Q^{-1}P^{3\varepsilon}\left\{1+P^{6(-\Lambda+t+n-1/2)/(n-1)}\right\}\ll 1,
\end{displaymath}
in view of (\ref{apr_use}) and the choice of $Q$. And
\begin{displaymath}
H^n R^{-n/2}P^{-(n-2)}\psi_H^{-(n-2)/2}\ll R^{-\frac{n}{2}}\phi^{-\frac{n-2}{2}}P^{-(n-2)+n \varepsilon}+ R^{-\frac{n}{2}+\frac{2n}{n-1}}\phi^{-\frac{n-2}{2}+\frac{n}{n-1}}P^{-(n-2)+\frac{2n(-\Lambda+t+n-1/2)}{n-1}+n\varepsilon}
\end{displaymath}
The exponent on $\phi$ in the second term is strictly negative for $n\geq 6$, hence the second term is $\textit{O}(1)$ provided that $\phi\geq \phi_1 P^\delta$. On assuming $\phi\geq \phi_1 P^\delta$, the first term is
\begin{displaymath}
R^{-\frac{n}{2}}\phi^{-\frac{n-2}{2}}P^{-(n-2)+n \varepsilon}\ll R^{-\frac{n(n-4)}{ n^2 -  5n +2}} P^{-\frac{ n (n-2 )(-2\Lambda+2 t+2n-3)}{  n^2 -  5n +2}}\ll 1,
\end{displaymath}
provided that $-2\Lambda+2t+2n-3\geq 0$ (noting that the exponent on $R$ is strictly negative when $n\geq 6$). When $-2\Lambda+2t+2n-3< 0$, one has $\pi_1< 2$. So $\phi\geq \phi_1 P^\delta \geq R^{-\rho_1}P^{-2+\delta}$ and
\begin{displaymath}
R^{-\frac{n}{2}}\phi^{-\frac{n-2}{2}}P^{-(n-2)+n \varepsilon}\ll R^{-\frac{n}{2}}\left(R^{-\rho_1}P^{\delta}\right)^{-\frac{n-2}{2}}P^{n \varepsilon}\ll R^{-\frac{n(n-4)}{n^2-5n+2}}\ll 1.
\end{displaymath}
Now we have $F\ll 1$ and
\begin{displaymath}
E_1:=P^{t+\varepsilon}\max\limits_{R,\phi}R\left(\frac{\psi_H P^{2n-1}}{H^{n-1}}F\right)^{1/2}\ll \max\limits_{R,\phi} \frac{R\phi^{1/2}P^{(2n-1)/2+t+\varepsilon}}{H^{(n-1)/2}}\ll P^{\Lambda-\frac{n-3}{2}\varepsilon}.
\end{displaymath}
Then Lemma \ref{CL1} follows.
\end{proof}
\begin{lemma} \label{CL2}
Suppose that (\ref{req0}), (\ref{req1}), (\ref{req3}), (\ref{req7}) holds. Then
\begin{displaymath}
I_2(S;t,\mathfrak{n}_2)=\textit{o}\left(P^\Lambda\right),
\end{displaymath}
where $\mathfrak{n}_2$ collects $\alpha\in \mathfrak{m}$  with the representation (\ref{apr}) with
\begin{displaymath}
P^{\Xi+\delta} \leq q\leq Q,\quad |\alpha-a/q|\leq \min\{\phi_0,(qQ)^{-1}\}.
\end{displaymath}
\end{lemma}
\begin{proof}
Similarly one has
\begin{displaymath}
I_2(S;t,\mathfrak{n})\ll P^{t+\varepsilon}\left\{P^{3/2}+\max\limits_{1\leq R\leq Q \atop \phi \leq \min\{\phi_0,\phi_1 P^{\delta}, \phi_2P^{-\delta}\}}R\left(\frac{\psi_H P^{2n-1}}{H^{n-1}}F\right)^{1/2}\right\}:=P^{t+3/2+\varepsilon}+E_2.
\end{displaymath}
And (\ref{req1}) ensures the first term on the right is $\textit{O}(P^{\Lambda})$. Take $H=\left\lfloor P^{4\varepsilon} \left(1+RP^{-\Lambda+t+n-3/2}\right)^{2/n}\right\rfloor$, then $(P^2H)^{-1}\ll \psi_H\ll (P^2H)^{-1}P^{4\varepsilon}$ for $\phi\leq \phi_0$. Combining (\ref{apr_use}), (\ref{req3}) and the choice of $Q$ gives
\begin{displaymath}
H\ll P^{4\varepsilon}\left\{1+\left(QP^{-\Lambda+t+n-3/2}\right)^{2/n}\right\}\leq P,
\end{displaymath}
i.e., the choice of $H$ is appropriate.
On assuming (\ref{req0}) and (\ref{req7}), one reaches
\begin{eqnarray}
RH^3\psi_H\ll RH^2P^{-2+4\varepsilon}\ll RP^{-2+12\varepsilon}\left\{1+\left(R P^{-\Lambda+t+n-3/2}\right)^{4/n}\right\} \nonumber\\
\ll QP^{-2+12\varepsilon}+Q^{1+\frac{4}{n}}P^{\frac{-4\Lambda+4t+2n-6}{n}+12\varepsilon} \label{reproof}\\
\ll 1+P^{\frac{(n+2) (-10 \Lambda + 10 t+ 8 n -3 )}{ n(n-1)}+18\varepsilon}\ll P^{18\varepsilon}. \nonumber
\end{eqnarray}
Moreover,
\begin{displaymath}
H^n R^{-n/2}P^{-(n-2)}\psi_H^{-(n-2)/2}\ll R^{-n/2}P^{(3n/2-1)4\varepsilon}+ R^{-n/2+3-2/n}P^{\left(-\Lambda+t+n-\frac{3}{2}\right)\left(3-\frac{2}{n}\right)+(3n/2-1)4\varepsilon}\ll 1,
\end{displaymath}
whenever $R\geq P^{\frac{(3n-2) ( - 2 \Lambda + 2 t+ 2 n -3)}{ n^2 - 6 n +4}+\delta}=P^{\Xi+\delta}$. (Note that the exponent on $R$ in the second term is strictly negative when $n\geq 6$.) Now we have $F\ll P^{18\varepsilon}$ and
\begin{displaymath}
E_2:=P^{t+\varepsilon}\max\limits_{R,\phi}R\left(\frac{\psi_H P^{2n-1}}{H^{n-1}}F\right)^{1/2}\ll \max\limits_{R,\phi} \frac{RP^{(2n-3)/2+t+12\varepsilon}}{H^{n/2}}\ll P^{\Lambda-(2n-10)\varepsilon}.
\end{displaymath}
Then Lemma \ref{CL2} follows.
\end{proof}
\begin{proof}[of Lemma \ref{LE8}]
First suppose that $R\geq P^{\Xi+c\delta}$, where $c=\max\left\{(\rho_1-\rho_0)^{-1},1\right\}$. Note that
\begin{displaymath}
\rho_1-\rho_0=\frac{(n-1)(n^2-6n+4)}{n(n^2-5n+2)}>0
\end{displaymath}
for $n\geq 6$, the positive constant $c$ depends only on $n$. One can check that $\phi_0\geq \phi_1 P^{\delta}$ under the condition $R\geq P^{\Xi+c\delta}$ and the choice of $c$. A combination of Lemma \ref{CL1} and Lemma \ref{CL2} shows that
\begin{displaymath}
I_1(S;t,\mathfrak{m}_1)=\textit{o}\left(P^{\Lambda}\right),
\end{displaymath}
where $\mathfrak{m}_1$ denotes the set of $\alpha\in\mathfrak{m}$ with the representation (\ref{apr}) with $q\geq P^{\Xi+c\delta}$.

For $R\leq P^{\Xi+c\delta}$, Lemma \ref{CL1} again shows that $I_1(S;t,\mathfrak{m}_2)=\textit{o}\left(P^{\Lambda}\right)$, where $\mathfrak{m}_2$ the $\alpha\in \mathfrak{m}$ with the representation (\ref{apr}) with $q\leq P^{\Xi+c\delta}, \quad |\beta|\geq \phi_1P^{\delta}$.

As for the remaining range, an application of Lemma \ref{LE3} yields
\begin{displaymath}
\begin{aligned}
P^{t+\varepsilon}\mathcal{M}_2(R,\phi,\pm)^{1/2}\ll &P^{t+\varepsilon}(R^2\phi)^{1/2}P^{n+\varepsilon}\left\{R\phi+(R\phi P^3)^{-1}\right\}^{n/8}\ll P^{n+t+2\varepsilon}R^{1+\frac{n}{8}}(\phi_1P^{c\delta})^{\frac{1}{2}+\frac{n}{8}}+ P^{\frac{5n}{8}+t+2\varepsilon}R^{1-\frac{n}{8}}\phi^{\frac{1}{2}-\frac{n}{8}}\\
\ll &R^{1+\frac{n}{8}-\rho_1\left(\frac{1}{2}+\frac{n}{8}\right)}P^{n+t-\pi_1\left(\frac{1}{2}+\frac{n}{8}\right)+\textit{O}(\delta)}+R^{1-\frac{n}{8}}\phi^{\frac{1}{2}-\frac{n}{8}}P^{\frac{5n}{8}+t+2\varepsilon}.
\end{aligned}
\end{displaymath}
The first term is $\textit{o}\left(P^\Lambda\right)$ when (\ref{req11}) holds. The second term is $\textit{o}\left(P^\Lambda\right)$ when $\phi\geq \phi_2 P^{\delta}$. Then (\romannumeral1) follows.

Now if $\Xi\leq 0$, then $q\leq P^{\Xi+c \delta}$ and $|\beta|\leq \phi_2 P^{\delta}$ implies $q\leq P^\Delta$ and $|\beta|\leq P^{-\pi_2+\textit{O}(\delta)}\leq P^{-3+\Delta}$ on assuming that (\ref{req12}) holds. Then $\alpha$ lies in the major arcs and $\mathfrak{m}_0=\emptyset$. Hence (\romannumeral2) follows.
\end{proof}
The treatment of (1,1,9)-forms needs a certain fourth moment of a cubic exponential sum. The following lemma is a slight modification of Br\"{u}dern \cite[Theorem 2]{Bru91}, which involves an application of a Kloosterman refinement based on \cite{Hoo86}.

Define the weight function
\begin{displaymath}
w(x):=
\begin{cases}
e^{-\frac{1}{1-x^2}},\quad &\text{if }|x|<1,\\
0, &otherwise.
\end{cases}
\end{displaymath}

Let
\begin{displaymath}
T(\alpha)=\sum\limits_{x\in\mathbb{Z}}w(P^{-1}x)e\left(\alpha x^3\right).
\end{displaymath}
Denote
\begin{displaymath}
\mathfrak{N}=\mathfrak{N}(R,\phi):=\bigcup\limits_{R<q\leq 2R}\bigcup\limits_{a=1 \atop (a,q)=1}^q\left[\frac{a}{q}+\phi, \, \frac{a}{q}+2\phi\right].
\end{displaymath}
\begin{lemma} \label{LE9}
For $\phi\leq P^{-3}$, we have
\begin{displaymath}
\int_{\mathfrak{N}(R,\phi)} |T(\alpha)|^4 d\alpha\ll P^{\varepsilon}\left(P^4 \phi +R^{7/2}\phi+R^2\phi P^2\right).
\end{displaymath}
And for $\phi>P^{-3}$, we have
\begin{displaymath}
\int_{\mathfrak{N}(R,\phi)} |T(\alpha)|^4 d\alpha\ll P^{\varepsilon}\left(\phi^{-1/3}+R^{7/2}\phi^3 P^6+R^2\phi P^2\right).
\end{displaymath}
\end{lemma}
The weight function used here is slightly different from that in \cite{Bru91}. It is actually the weight in \cite{Hoo86}. However, the validity of the argument is not affected.

\section{Bounding the minor arc estimates}

In this section, we bound the minor arc estimates and prove Proposition \ref{PR2}. Recall the definition of $\mathfrak{m}_0$ and $\mathcal{A}$. They collect the $\alpha\in [0,1]$ with the representation (\ref{apr}) with
\begin{displaymath}
\begin{aligned}
&\mathfrak{m}_0\subseteq \mathfrak{m}: &&\quad q\leq P^{\Xi+\textit{O}(\delta)},\quad |\beta|\leq \phi_2P^{\delta}=R^{-\rho_2}P^{-\pi_2+\textit{O}(\delta)},\\
&\mathcal{A}(A,B,C): &&\quad q\leq P^A,\quad |\beta|\leq q^{-B}P^{-3+C}.
\end{aligned}
\end{displaymath}
\begin{proof}[of the case $(n_1,n_2,n_3)=(1,2,8)$]
Recall $S_i(\alpha)=S(\alpha; C_i,n_i,\rho,P)$ $(0\leq i\leq 3)$. Lemma \ref{LE5} shows that
\begin{displaymath}
I_{4}(S_1;0,[0,1])\ll P^{1/2+\varepsilon}, \quad I_{4}(S_2;0,[0,1])\ll P^{5/4+\varepsilon}.
\end{displaymath}
Taking $n=8,v=2,t=7/4,\Lambda=8$, Lemma \ref{LE8}(\romannumeral1) gives $\Xi =11/20, \rho_2=0,\pi_2=5/2$ and
\begin{displaymath}
I_1(S_1S_2S_3;0,\mathfrak{m}\setminus \mathfrak{m}_0)\ll I_{4,4,2}(S_1,S_2,S_3;0,\mathfrak{m}\setminus \mathfrak{m}_0)\ll I_{2}(S_3;7/4,\mathfrak{m}\setminus \mathfrak{m}_0)=\textit{o}\left(P^8\right).
\end{displaymath}
Noting that $\mathfrak{m}_0\subseteq \mathcal{A}\left(11/20+\textit{O}(\delta),0,1/2+\textit{O}(\delta)\right):=\mathcal{A}$. The $\textit{O}(\delta)$-term occurring here can be as small as we want, so substituting $\textit{O}(\delta)$ for $\varepsilon$ does not affect the validity of the lemmas. Lemma \ref{LE6} gives
\begin{displaymath}
S_{1}(\alpha)=S_{1}^\ast (\alpha)+\textit{O}\left(P^{21/40+\textit{O}(\delta)}\right), \quad I_{4}(S_1^\ast ;0,[0,1])\ll P^{1/4+\varepsilon}.
\end{displaymath}
for $\alpha\in \mathcal{A}$. And
\begin{displaymath}
I_1(S_1S_2S_3;0,\mathfrak{m}_0)\ll I_1(S_1^\ast S_2 S_3;0,\mathfrak{m})+I_1(S_2S_3;21/40,\mathcal{A}).
\end{displaymath}
Taking $n=8,v=2,t=1/4+5/4=3/2, \Lambda=8$, Lemma \ref{LE8}(\romannumeral2) shows that $\Xi=0,\rho_2=0,\pi_2=3$ and
\begin{displaymath}
I_1(S_1^\ast S_2 S_3;0,\mathfrak{m})\ll I_{4,4,2}(S_1^\ast,S_2,S_3;0,\mathfrak{m})\ll I_{2}(S_3;3/2; \mathfrak{m})=\textit{o}\left(P^8\right).
\end{displaymath}
Denote $\widetilde{S}(\alpha):=S(\alpha;C_2+C_3,10,\rho,P)$. It follows from Lemma \ref{LE7} that
\begin{displaymath}
I_1(S_2S_3;21/40,\mathcal{A})=I_1(\widetilde{S};21/40,\mathcal{A})\ll P^{8-1/16+\textit{O}(\delta)}.
\end{displaymath}
To conclude, we have $I_1(S_1S_2S_3;0,\mathfrak{m})=\textit{o}\left(P^8\right)$. Proposition \ref{PR2} follows and $\mathcal{N}(P)\gg P^8$.
\end{proof}
\begin{proof}[of the case $(n_1,n_2,n_3)=(1,1,9)$]
Define, for $i=1,2$,
\begin{displaymath}
T_i(\alpha)=\sum\limits_{x\in\mathbb{Z}}w\left((\rho P)^{-1}x\right)e\left(\alpha C_{i}(x)\right).
\end{displaymath}
The parameters $\rho\leq 1$ is determined in Lemma \ref{LE2}. It is fixed and makes no difference to the validity of applying Lemma \ref{LE9}. Let $\mathfrak{b}\subseteq \mathfrak{a}\subseteq \mathfrak{m}$ be the set of $\alpha \in \mathfrak{m}$ with the representation (\ref{apr}) with
\begin{displaymath}
\begin{aligned}
&\mathfrak{a}: &&\quad q\leq P^{25/31+\textit{O}(\delta)},\quad |\beta|\leq q^{-1/5}P^{-11/5+\textit{O}(\delta)}\\
&\mathfrak{b}: &&\quad q\leq P^{1145/1922+\textit{O}(\delta)},\quad |\beta|\leq q^{-1/5}P^{-1867/775+\textit{O}(\delta)}.
\end{aligned}
\end{displaymath}

It is easy to see that, for $i=1,2$, one has
\begin{displaymath}
I_{4}(T_i,0,[0,1])\ll P^{1/2+\varepsilon}
\end{displaymath}
by regarding the left side as the weighted number of solutions to the equality
\begin{displaymath}
C_i(x_1)+C_i(x_2)=C_i(x_3)+C_i(x_4).
\end{displaymath}
Taking $n=9,v=2,t=1/2+1/2=1,\Lambda=8$, Lemma \ref{LE8} yields $\Xi =25/31, \rho_2=1/5,\pi_2=11/5$ and
\begin{displaymath}
I_1(T_1T_2S_3;0,\mathfrak{m}\setminus \mathfrak{a})\ll I_{4,4,2}(T_1,T_2,S_3;0,\mathfrak{m}\setminus \mathfrak{a})\ll I_{2}(S_3;1,\mathfrak{m}\setminus \mathfrak{a})=\textit{o}\left(P^8\right). \label{case1-1}
\end{displaymath}
For $R,\phi$ in the range defined by $\mathfrak{a}$, Lemma 9 gives
\begin{displaymath}
I_4(T_1;0,\mathfrak{a})\ll P^{\varepsilon} \max_{R,\phi} I_4(T_1;0,\mathfrak{N}(R,\phi))\ll P^{\frac{539}{310}\cdot \frac{1}{4}+\textit{O}(\delta)}= P^{\frac{539}{1240}+\textit{O}(\delta)}.
\end{displaymath}
Taking $n=9,v=2,t=539/620,\Lambda=8$, Lemma \ref{LE8} yields $\Xi =1145/1922, \rho_2=1/5,\pi_2=1867/775$ and
\begin{displaymath}
I_1(T_1T_2S_3;0,\mathfrak{a}\setminus \mathfrak{b})\ll I_{4,4,2}(T_1,T_2,S_3;0,\mathfrak{a}\setminus \mathfrak{b}) \ll I_{2}(S_3;539/620,\mathfrak{m}\setminus \mathfrak{b})=\textit{o}\left(P^8\right). \label{case1-1}
\end{displaymath}
For $R,\phi$ in the range defined by $\mathfrak{b}$, Lemma 9 again gives
\begin{displaymath}
I_4(T_1;0,\mathfrak{b})\ll P^{\varepsilon} \max_{R,\phi} I_4(T_1;0,\mathfrak{N}(R,\phi))\ll P^{\frac{1}{4}+\textit{O}(\delta)}.
\end{displaymath}
And noting that $\mathfrak{b}\subseteq \mathcal{A}(1145/1922+\textit{O}(\delta), 1/5, 458/775+\textit{O}(\delta)):=\mathcal{A}$. By (\ref{remarkoflemma7}), it follows that
\begin{displaymath}
I_1(T_1 T_2 S_3;0,\mathfrak{b})\ll I_{4,4,2}(T_1,T_2,S_3;0,\mathfrak{b})\ll I_{2}(S_3;1/2;\mathcal{A}\cap \mathfrak{m})\ll P^{8-\Delta/8+\textit{O}(\delta)}.
\end{displaymath}
To sum up, we now have
\begin{displaymath}
I_1(T_1T_2S_3;0,\mathfrak{m})=\textit{o}\left(P^8\right).
\end{displaymath}

The weight function $w$ satisfies $w\geq 0$ in $\mathbb{R}$ and $w(x)\gg 1$ for $|x|<\rho P/2$, which ensures the argument in the treatment of singular integral in Lemma \ref{LE2}. We conclude that
\begin{displaymath}
\mathcal{N}(P)\gg P^8.
\end{displaymath}
\end{proof}


\section{Further remarks}

The $(1,1,8)$ and $(1,2,7)$ cases are hard to solve. The estimates on exponential sums of the $1$-forms and $2$-forms are not small enough and Heath-Brown's $L^2$ bound can not be used. Neither can it even in the case $(1,1,1,7)$. We say $C_0$ splits into four forms, and is a $(n_1,n_2,n_3,n_4)$-form, if
\begin{displaymath}
C_0(\mathbf{x})=C_1(\mathbf{x}_1)+C_2(\mathbf{x}_2)+C_3(\mathbf{x}_3)+C_4(\mathbf{x}_4),
\end{displaymath}
where $C_i\in \mathbb{Z}[\mathbf{x}_i]\,(1\leq i\leq 4)$ are cubic forms in $n_i$ variables and $n_1+n_2+n_3+n_4=n_0$.
 For $n_0=10$, according to (\ref{cases}), the only case not solved is $(1,1,1,7)$.
  We need strong hypothesis, such as Hypothesis $HW_6$ (see \cite[p.10]{HB98} for details), to solve the $(1,1,1,7)$ case.
   This hypothesis involves Riemann Hypothesis and standard analytic continuation of certain Hasse-Weil $L$-functions. We record the following proposition.
\begin{proposition} \label{TH6}
Let $X\subseteq \mathbb{P}^{n-1}$ be a hypersurface defined by a cubic form that splits into four forms, with $n\geq 10$. Assuming Hypothesis $HW_6$, we have $X(\mathbb{Q})\neq \emptyset$.
\end{proposition}

The singular series $\mathfrak{S}$ is still absolutely convergent in the $(1,1,1,7)$ case (we have $S_{a,q}(C_i)\ll q^{2/3}$ for $1\leq i\leq 3$, which is better than $q^{5/6}$). The $(1,1,8)$ case remains unproved under this strong hypothesis. So it seems really hard to improve $n\geq 11$ to $n\geq 10$ in Theorem \ref{TH1}.

\begin{acknowledgements}\label{ackref}
We would like to thank Professor T. D. Browning for giving talks on this topic and suggesting this problem. He has also provided us with useful advices and great help, including pointing out the improved version of \cite[Theorem 2]{Browning1}, Theorem H, \cite[Proposition 2]{HB83} and Br\"{u}dern's fourth moment estimate. We also thank Yun Gao for her kind help. The first author is especially grateful to his supervisor, Professor Hongze Li, for his unending encouragement and his effort that provides the opportunities to communicate with foreign specialists.
\end{acknowledgements}

\affiliationone{
   Boqing Xue and Haobo Dai\\
   800 Dongchuan Road, Department of Mathematics, Shanghai Jiao Tong University, Shanghai\\
   China
   \email{ericxue68@sjtu.edu.cn\\
   dedekindbest@hotmail.com}}

\end{document}